\documentclass{amsart}

\usepackage{cite}
\usepackage{amssymb}
\usepackage{graphicx}
\usepackage[cmtip,all]{xy}
\usepackage{MnSymbol}

\newtheorem{theorem}{Theorem}[section]
\newtheorem{lemma}[theorem]{Lemma}
\newtheorem{prop}[theorem]{Proposition}
\newtheorem{cor}[theorem]{Corollary}

\theoremstyle{definition}

\theoremstyle{remark}
\newtheorem{remark}[theorem]{Remark}

\numberwithin{equation}{section}

\newcommand{\abs}[1]{\left\lvert#1\right\rvert}

\newcommand{\ds}{\displaystyle}

\newcommand{\CB}{\mathcal{B}}
\newcommand{\CC}{\mathcal{C}}

\newcommand{\CE}{\mathcal{E}}

\newcommand{\CH}{\mathcal{H}}
\newcommand{\CI}{\mathcal{I}}
\newcommand{\CJ}{\mathcal{J}}

\newcommand{\CM}{\mathcal{M}}

\newcommand{\CY}{\mathcal{Y}}

\begin{document}

\title[Distribution of values of short hybrid exponential sums II]
{The distribution of values of short hybrid exponential sums on curves over finite fields II}

\author{Kit-Ho Mak}

\subjclass[2010]{Primary 11G20, 11T23, 11T24}
\keywords{Gaussian distribution, hybrid exponential sums, algebraic curves}

\begin{abstract}
Let $p$ be a prime number, $C$ be any absolutely irreducible affine plane curve over $\mathbb{F}_p$, $g,f\in\mathbb{F}_p(x,y)$, and $\mathcal{J}$ be an interval. We continue the study of the distribution of the values of short hybrid exponential sums of the form
\begin{equation*}
S_{H}(x;C) = \sum_{\substack{P\in C, x<x(P)\leq x+H \\ y(P)\in\mathcal{J}}}\chi(g(P))\psi(f(P))
\end{equation*}
on $x\in\mathcal{I}$ for some short interval $\mathcal{I}$. We show that under some natural conditions, the limiting distribution of the sum $S_{H}(x;C)$ is Gaussian for all curve $C$. This largely generalizes a previous result of the author and Zaharescu.
\end{abstract}

\maketitle

\section{Introduction}

One of the main theme of research in analytic number theory is to understand the distribution of short character sums and exponential sums. Let $p$ be a prime, let $\chi$ be a multiplicative character modulo $p$, and let $\psi$ be an additive character modulo $p$. Let $f,g\in\mathbb{F}_p(x)$ be two rational functions. The short hybrid exponential sum is defined by
\begin{equation}\label{eqnshorthybsum}
S_{H}(x)=\sum_{x<n\leq x+H}\chi(g(n))\psi(f(n)),
\end{equation}
where $H=H(p)\leq p$. We employ the convention that all sums in our paper will exclude the poles of $f$ and $g$. When $\chi$ is the quadratic character, $g(x)=x$ and $\psi$ is trivial, Davenport and Erd\"{o}s \cite{DaEr52} showed that the resulting character sum
\begin{equation}\label{eqnshortsum}
S_{H}(x)=\sum_{x<n\leq x+H}\chi(n)
\end{equation}
tends to a Gaussian distribution with mean zero and variance $H$ when $p$ tends to infinity, provided that $H\rightarrow\infty$  and $\log{H}/\log{p}\rightarrow 0$ as $q\rightarrow\infty$. More precisely, they showed that under such conditions,
\begin{equation*}
\lim_{p\rightarrow\infty}\frac{1}{p}\abs{\{ 0\leq x\leq p-1: S_H(x) \leq \lambda\sqrt{H} \}}=\frac{1}{\sqrt{2\pi}}\int_{-\infty}^{\lambda}e^{-t^2/2}\,dt.
\end{equation*}
In \cite{CCZ03}, the result of Davenport and Erd{\"o}s is generalized to the case of an $n$-dimensional sum of quadratic characters of the form
\begin{equation*}
S_H(x_1,\ldots,x_n)=\sum_{x_1<z_1\leq x_1+H}\cdots\sum_{x_n<z_n\leq x_n+H}\chi(z_1+\ldots+z_n),
\end{equation*}
and obtained a Gaussian distribution in that case. Lamzouri \cite{Lam13} studied the sum \eqref{eqnshortsum} when $\chi$ is non-real, and obtained a two dimensional Gaussian distribution for such sums.

In a previous paper \cite{MaZa11}, the author and Zaharescu investigated a more general type of short hybrid exponential sums over a plane curve. Let $p$ be a prime, and let $C$ be an absolutely irreducible affine curve over $\mathbb{F}_p$ of degree $D$, defined by the equation $P(x,y)=0$ with $\deg_yP(x,y)\geq 1$, where $\deg_y$ denotes the degree in $y$. Let $\chi$ be a multiplicative character and $\psi$ be an additive character modulo $p$. Let $f,g\in\mathbb{F}_p[x,y]$ be two rational functions on $C$, and let and $\mathcal{J}=[\alpha p,\beta p)$ be an interval ($0\leq \alpha<\beta \leq 1$). Define the short hybrid exponential sum over $C$ by
\begin{equation}\label{eqnshortCsum}
S(x) = S_{H}(x;C) := \sum_{\substack{P\in C, x<x(P)\leq x+H \\ y(P)\in\mathcal{J}}}\chi(g(P))\psi(f(P)).
\end{equation}
Note that when $C$ is the affine line $y=0$, the above sum \eqref{eqnshortCsum} reduces to the hybrid sum \eqref{eqnshorthybsum}. The author and Zaharescu proved that the distribution of $S(x)$ tends to a Gaussian distribution when projected to any line through the origin, under some natural assumptions (to ensure the sum $S(x)$ is not trivial) on $\chi,\psi,f,g$, and under the condition on $C$ that no two points on $C$ with $y$-coordinates in $\CJ$ can share a common $x$-coordinates. More precisely, let $\CI$ be an interval of length at least $p^{1/2+\epsilon}$. Set
\begin{equation}\label{def:ux}
u_{\theta}(x)=\frac{S(x)e^{-i\theta}+\overline{S(x)}e^{i\theta}}{2((\beta-\alpha)H)^{1/2}},
\end{equation}
which is the (suitably normalized) projection of $S(x)$ on the line $y=e^{i\theta}x$, and let $G_{p,\theta}(\lambda)$ be the number of $x$ with $u_{\theta}(x)\leq\lambda$, then 
\begin{equation*}
\lim_{p\rightarrow\infty} \frac{G_{p,\theta}(\lambda)}{\abs{\CI}}=\frac{1}{\pi}\int_{-\infty}^{\lambda}e^{-t^2}\,dt
\end{equation*}
when $S(x)$ is not real. Note that here the normalization required is different from that of Davenport-Erd\"{o}s. When $S(x)$ is real then we obtained the same Gaussian distribution as in the case for Devenport-Erd\"{o}s. 

Although the work \cite{MaZa11} generalized the result of Davenport-Erd\"{o}s to a large extend, the condition that no two points on $C$ with $y$-coordinates in $\CJ$ can share a common $x$-coordinates is a serious restriction. This restriction essentially says that the results in \cite{MaZa11} are only valid when $C$ is a rational curve or a hyperelliptic curve (and $\CJ$ sits inside $[0,(p-1)/2]$ or $[(p+1)/2,p-1]$). It is the aim of this paper to remove this restriction and obtain Gaussian distributions for short hybrid exponential sums of the form \eqref{eqnshortCsum} over \textit{all} plane curves. We will also generalize the arguments in \cite{MaZa11,Lam13} to study the two dimensional distribution of such sums on the complex plane. Note that our results on the two-dimensional distribution are new as long as $\psi$ is non-trivial or if $f(n)\neq n$, even when $C$ is the affine line. The new idea here is to introduce certain families of curves related to the exponential sums, and study their properties as a fibration using deeper tools from algebraic geometry. This enables us to study the sums \eqref{eqnshortCsum} in general.

\section{Statements of Main Results}

Let $p$ be a large prime, and $C$ be an absolutely irreducible affine plane curve over $\mathbb{F}_p$ defined by the equation $P(x,y)=0$ that is not a vertical line (i.e. the line defined by $x=0$). Let $D$ be the degree of $C$. Let $\chi$ be a multiplicative character and let $\psi$ be an additive character modulo $p$, not both trivial. Let $f,g\in\mathbb{F}_p(x,y)$ be two rational functions on $C$. Let $\CI\subseteq [0,p-1]$ and $\CJ=[\alpha p,\beta p)$ be intervals, where $0\leq \alpha < \beta \leq 1$. We also let $H$ be an integer such that $1\leq H \leq p$. Since $C$ is irreducible, a standard argument using completions of exponential sums shows that the number of points $N$ on $C$ inside the rectangle $(x,x+H]\times\CJ$ is given by (see for example \cite{MaZa12})
\begin{equation*}
N = (\beta-\alpha)H + O(\sqrt{p}\log^2{p}).
\end{equation*}
We are interested in the distribution of the values of the hybrid exponential sums \eqref{eqnshortCsum}, namely,
\begin{equation*}
S(x) = S_{H}(x;C) := \sum_{\substack{P\in C, x<x(P)\leq x+H \\ y(P)\in\mathcal{J}}}\chi(g(P))\psi(f(P)),
\end{equation*} 
for $x\in\CI$ as $p$ tends to infinity. It is understood that the poles of $f,g$ are excluded from the sum.

As observed in the case when $C$ is the hyperelliptic curve \cite{MaZa11}, the distribution of $S(x)$ depends on whether it is complex (i.e. non-real) or real. If the sum $S(x)$ is non-trivial, it is complex unless $\psi$ is trivial and $\chi$ is quadratic. We are able to show that the distribution of $S(x)$ is Gaussian in both cases. We will deal with the complex case first, and show that $S(x)$ has a two-dimensional Gaussian distribution.

\begin{theorem}\label{thm1}
Let $p$ be a prime. Let $C$ be an absolutely irreducible curve over $\mathbb{F}_p$ defined by the equation $P(x,y)=0$. Let $g,f\in\mathbb{F}_p(x,y)$ be two rational functions. Let $\chi$, $\psi$, $\CI$, $\CJ$, $H$ be as above. 

If $\psi$ is nontrivial, we assume $f$ is nonlinear on $C$, i.e. $f$ is not of the form
\[\text{(linear terms)}+h_2(x,y)P(x,y)^b\]
for any nonzero integer $b$, rational functions $h_1\in\overline{\mathbb{F}}_p(x,y), h_2\in\mathbb{F}_p(x,y)$, with $h_2$ relatively prime to $P$ (in this paper, all ``linear terms'' have coefficients in $\mathbb{F}_p$). 
Let $f=\frac{f_1}{f_2}$ with $f_1,f_2\in\mathbb{F}_p[x,y]$, where $f_1,f_2$ have no common factors and both are of degree less than $p$. Assume that
\begin{enumerate}
\item if $f$ is a polynomial, then write $f(x,y)=r_1(x)+r_2(x,y)$, where $r_1$ consists of all terms which do not depend on $y$. We further assume that either
    \begin{enumerate}
    \item $r_2$ is nonlinear on $C$, or
    \item if $r_2$ is linear on $C$, we assume that $\deg{r_1}>2$.
    \end{enumerate}
\item if $f$ is not a polynomial, i.e. $\deg{f_2}\geq 1$, then assume that $\deg{f_2}=o(\log{p})$ is small.
\end{enumerate}

On the other hand, if $\psi$ is trivial, then we assume that $\chi$ has order $a>2$ (so that $S(x)$ is complex). Assume $g(x,y)$ is not of the form
\begin{equation}\label{eqnthm21} 
h_1^a+h_2(x,y)P(x,y)^b
\end{equation}
for any nonzero integer $b$, $h_2\in\mathbb{F}_p(x,y)$ relatively prime to $P$ and $h_1\in\overline{\mathbb{F}}_p(x,y)$, and if $a$ is even, we further assume that $g(x,y)$ is not a complete $(a/2)$-th power on $C$ (i.e. \eqref{eqnthm21} cannot hold with $a/2$ in place of $a$).

Fix an $R\subseteq \mathbb{C}$ to be a rectangle whose edges are parallel to the coordinate axes (so that $R$ does not change with $p$). If 
\begin{equation*}
\liminf_{p\rightarrow\infty}\frac{\log{\abs{\CI}}}{\log{p}}>\frac{1}{2},
\end{equation*}
$\log{H}=o(\log{p})$, and $H\rightarrow\infty$ as $p\rightarrow\infty$, then we have
\begin{equation*}
\lim_{p\rightarrow\infty}\frac{1}{\abs{\CI}}\abs{ \left\{x\in\CI: \frac{S(x)}{\sqrt{H(\beta-\alpha)/2}}\in R \right\}} = \frac{1}{2\pi}\int\int_R e^{-\frac{x^2+y^2}{2}}\,dxdy.
\end{equation*}
\end{theorem}

Let $u_{\theta}(x)$ be the normalized projection of $S(x)$ onto the line $y=e^{i\theta}x$ as defined in \eqref{def:ux}. We can recover the main theorems in \cite{MaZa11} about the distribution of $u_{\theta}(x)$ immediately from Theorem \ref{thm1}, and show that the same distribution holds over \textit{all} (absolutely irreducible) curves.
\begin{cor}\label{cor1}
Notations and assumptions are as in Theorem\ref{thm1}. For any $\lambda\geq 0$, let $G_p(\lambda)$ be the number of $x\in\CI$ with $u_{\theta}(x)\leq\lambda$. We have 
\begin{equation*}
\lim_{p\rightarrow\infty}\frac{G_p(\lambda)}{\abs{\CI}}=\frac{1}{\sqrt{\pi}}\int_{-\infty}^{\lambda}e^{-t^2}\,dt
\end{equation*}
for any $\theta$.
\end{cor}

For the real case, the distribution of $S(x)$ is also (one-dimensional) Gaussian. As in the case for Davenport-Erd\"{o}s, the normalization required is different from the complex case.
\begin{theorem}\label{thm2}
Let $p$ be a prime. Let $C$ be an absolutely irreducible curve over $\mathbb{F}_p$ defined by the equation $P(x,y)=0$. Let $g\in\mathbb{F}_p(x,y)$ be a rational function. Assume that $\CI$, $\CJ$, $H$ satisfy the assumptions in Theorem \ref{thm1}, and that $S(x)$ is real, i.e. $\psi$ is trivial and $\chi$ is the quadratic character. Suppose $g(x,y)$ is not of the form
\[\ds h_1^2+h_2(x,y)P(x,y)^b \]
for any nonzero integer $b$, $h_2\in\mathbb{F}_p(x,y)$ relatively prime to $P$ and $h_1\in\overline{\mathbb{F}}_p(x,y)$. For any $\lambda\geq 0$, let $G_p(\lambda)$ be the number of $x\in\CI$ with $S(x)\leq\lambda(H(\beta-\alpha))^{1/2}$. We have
\begin{equation*}
\lim_{p\rightarrow\infty}\frac{G_p(\lambda)}{\abs{\CI}}=\frac{1}{\sqrt{2\pi}}\int_{-\infty}^{\lambda}e^{-\frac{t^2}{2}}\,dt.
\end{equation*}
\end{theorem}

The proof of these theorems will be based on estimates for various moments attached to $S(x)$ (see Theorem \ref{thm41} and Theorem \ref{thm42}). In the course of estimation we will need to understand the algebraic geometry of some curves related to $C$ (see Section \ref{sec5}). These moment estimates may be of independent interests. We will end this section with several remarks discussing several aspects of the theorems.

\begin{remark}
As in the case when $C$ is the affine line, Theorem \ref{thm1} and \ref{thm2} is in some sense the best possible with respect to the range of $\abs{\CI}$, and $f$ has to be non-linear. See \cite[Example 2.6, 2.7]{MaZa11} for counterexamples of Corollary \ref{cor1} when these conditions are violated. These are also counterexamples to Theorem \ref{thm1} and \ref{thm2} when the conditions are violated.
\end{remark}

\begin{remark}
On the other hand, the assumption that $C$ is absolutely irreducible can be removed provided that $g$, $f$ satisfy the assumptions of Theorem \ref{thm1} and \ref{thm2} for each irreducible component of $C$ that are defined over $\mathbb{F}_p$. We outline the argument here. If $C$ is any plane curve, let $\Gamma_1,\ldots,\Gamma_t$ be the absolutely irreducible components of $C$. Some of the $\Gamma_i$ are defined over $\mathbb{F}_p$ and some may not. For those $\Gamma_i$ defined over $\mathbb{F}_p$, our theorems apply and we get a Gaussian distribution for $S_H(x;\Gamma_i)$. For those components that are not defined over $\mathbb{F}_p$, it is well-known that those components have only finitely many points (i.e. $O(1)$ as $p$ tends to infinity) over $\mathbb{F}_p$. So these components do not affect the overall distribution as $p$ tends to infinity. Since any two components intersect only at finitely many points, we have $S_H(x;C)=\sum_i S_H(x;\Gamma_i)+O(1)$. Thus we get a Gaussian distribution for $S_H(x;C)$.
\end{remark}

\begin{remark}
It is also interesting to note that for any $x$, the number of summands in $S(x)$ is the number of points $N_{\CB_x}(C)$ on $C$ inside the rectangle $\CB_x=(x,x+H]\times\CJ$. It is very possible that the distribution of $N_{\CB_x}(C)$ is itself Gaussian, and this is known to be true in some special cases, including the modular hyperbola $xy=1$ \cite{GKS02} and hyperelliptic curves \cite{Mak14}. We hope that the ideas in this paper can shed some light to the problem of the distribution of $N_{\CB_x}(C)$, and more generally the distribution of points on curves inside rectangles that are too small for the Weil bound to be useful.
\end{remark}

\section{Some preliminary lemmas}

In this section, we collect together several lemmas that will be used later. The first lemma is an estimate for incomplete hybrid exponential sums over space curves is crucial in our argument. It is important to note that the lemma works even when the underlying curve is not absolutely irreducible.
\begin{lemma}\label{lem:perel}
Let $p$ be a large prime, and let $Y$ be an affine curve of degree $D$. Let $\tilde{g},\tilde{f}\in\mathbb{F}_p(x_1,\ldots,x_m)$ be two rational functions, and let $d_{\tilde{g}},d_{\tilde{f}}$ be the degrees of the denominators of $\tilde{g}, \tilde{f}$ respectively. Let $\chi$ be a multiplicative character and $\psi$ be an additive character, not both trivial. Let $\CJ_1,\ldots,\CJ_m\subseteq[0,p-1]$ be intervals, and define the hybrid exponential sum $S_{\CJ_1,\ldots,\CJ_m}$ by
\begin{equation*}
S_{\CJ_1,\ldots,\CJ_m} = \sum_{\textbf{x}\in Y\cap(\CJ_1\times\ldots\times\CJ_m)}\chi(\tilde{g}(\textbf{x}))\psi(\tilde{f}(\textbf{x})),
\end{equation*}
where $\textbf{x}=(x_1,\ldots,x_m)$. Let $a$ be the order of $\chi$. Suppose one of the following two conditions is satisfied:
\begin{enumerate}
\item $\chi$ is non-trivial, and there is no rational function $\tilde{g_1}\in\overline{\mathbb{F}}_p(x_1,\ldots,x_m)$ such that $\tilde{g}-\tilde{g_1}^a$ vanishes identically on some irreducible component of $Y$; or
\item $\psi$ is non-trivial, and there is no rational function $\tilde{f_1}\in\overline{\mathbb{F}}_p(x_1,\ldots,x_m)$ such that  $\tilde{f}-\tilde{f_1}^p+\tilde{f_1}$ is linear on some irreducible component of $Y$.
\end{enumerate}
Then we have
\begin{equation*}
\abs{S_{\CJ_1,\ldots,\CJ_m}}\ll D(D+d_{\tilde{g}}+d_{\tilde{f}})\sqrt{p}\log^m{p}.
\end{equation*}
\end{lemma}
\begin{proof}
The case when all the $\CJ_j$'s are all full intervals is the main result in Perel'muter \cite{Per69}, which uses the idea of Bombeiri-Weil type estimate on exponential sums along a curve \cite{Wei48,Bom66}. The lemma then follows from a standard completion argument. See \cite[Lemma 3.1]{MaZa11} for details.
\end{proof}

Next, we will introduce a probability model and investigate its properties. Let $X_1,\ldots,X_H$ be independent random variables uniformly distributed on the unit circle, and let $Z_H=X_1+\ldots+X_H$. The next lemma is about the expected value $E((\text{Re}Z_H)^r(\text{Im}Z_H)^s)$.
\begin{lemma}\label{lem31}
Let $r,s\geq 0$ be integers. When $r+s$ is odd, we have
\begin{equation*}
E((\text{Re}Z_H)^r(\text{Im}Z_H)^s) = 0,
\end{equation*}
and when $r+s=2t$ is even, we have
\begin{equation*}
E((\text{Re}Z_H)^r(\text{Im}Z_H)^s) = t!(H^t+O(t^2 H^{t-1}))\sum_{\substack{0\leq j\leq r, 0\leq l\leq s \\ j+l=t}} \binom{r}{j}\binom{s}{l}(-1)^{s-l}.
\end{equation*}
\end{lemma}
\begin{proof}
Consider
\begin{align*}
E((\text{Re}Z_H)^r(\text{Im}Z_H)^s) &= E\left( \left(\frac{Z_H+\overline{Z_H}}{2}\right)^r \left(\frac{Z_H-\overline{Z_H}}{2i}\right)^s \right) \\
&= E\left( \sum_{j=0}^r\frac{1}{2^r}\binom{r}{j}Z_H^j\overline{Z_H}^{r-j} \sum_{l=0}^s\frac{1}{(2i)^s}\binom{s}{l}Z_H^l\overline{Z_H}^{s-l} \right) \\
&= \sum_{j=0}^r\sum_{l=0}^s\frac{1}{2^r(2i)^s}\binom{r}{j}\binom{s}{l}(-1)^{s-l}E(Z_H^{j+l}\overline{Z_H}^{r+s-(j-l)}).
\end{align*}
Note that
\begin{align*}
E(Z_H^a\overline{Z_H}^b) &= E\left( \left(\sum_{j=1}^H X_j\right)^a \left(\sum_{l=1}^H \overline{X_l}\right)^b \right) \\
&= \sum_{1\leq j_1,\ldots,j_a\leq H}\sum_{1\leq l_1,\ldots,l_b\leq H}E(X_{j_1}\ldots X_{j_a}\overline{X_{l_1}\ldots X_{l_b}}).
\end{align*}
Since $E(X_i)=E(\overline{X_i})=0$ and the $X_i$'s are independent, the above expectation is zero unless $a=b$ and $j_1,\ldots,j_a$ is a permutation of $l_1,\ldots,l_b$. The lemma now follows from noting that the number of combinations $1\leq j_1,\ldots, j_a,l_1,\ldots,l_a \leq H$ such that $j_1,\ldots,j_a$ being a permutation of $l_1,\ldots,l_a$ is 
\begin{equation*}
a!(H^a+O(a^2 H^{a-1})).
\end{equation*}
\end{proof}

We will also need the normalized version of the above model. Let
\begin{equation*}
\tilde{Z_H} = \frac{Z_H}{\sqrt{H/2}}.
\end{equation*}
Clearly we have
\begin{equation}\label{eqnlem32}
E((\text{Re}\tilde{Z_H})^r(\text{Im}\tilde{Z_H})^s) = \frac{1}{(H/2)^{(r+s)}/2} E((\text{Re}Z_H)^r(\text{Im}Z_H)^s)
\end{equation}
The following lemma, taken from \cite[Lemma 3.2]{Lam13}, shows that the joint distribution of $\text{Re}\tilde{Z_H}$ and $\text{Im}\tilde{Z_H}$ is close to the characteristic function of a two-dimensional standard Gaussian distribution.
\begin{lemma}\label{lem32}
Let $u,v$ be real numbers such that $\abs{u},\abs{v}\leq H^{1/4}$. Then
\begin{equation*}
E(e^{iu\text{Re}\tilde{Z_H}+iv\text{Im}\tilde{Z_H}}) = e^{-\frac{u^2+v^2}{2}}\left( 1+O\left(\frac{u^4+v^4}{H}\right) \right).
\end{equation*}
\end{lemma}

\section{The moments $M(r,s)$ and $M(k)$}

In this section we will define the moments and start our calculation of these moments. For any non-negative integers $r, s$, we define
\begin{equation*}
M(r,s):=\sum_{x\in\CI}(\text{Re}S(x))^r(\text{Im}S(x))^s,
\end{equation*}
and
\begin{equation*}
M(k):=M(k,0)=\sum_{x\in\CI}S(x)^k.
\end{equation*}
Note that if $S(x)$ is real, then $M(r,s)=0$ when $s>0$. Recall that $D=\deg{C}$. For any rational function $f\in\mathbb{F}_p(x)$, we define $d_g$, $d_f$ to be the degree of the denominator of $g$, $f$ respectively. Set $d=D+d_g+d_f$. Our estimates for the moments are the following.

\begin{theorem}\label{thm41}
Let $S(x)$ be complex and satisfies the assumptions in Theorem \ref{thm1}. Then we have
\begin{multline*}
M(r,s) = \frac{t!H^t\abs{\CI}(\beta-\alpha)^t}{2^{r+s}i^s}\sum_{\substack{0\leq j\leq r, 0\leq l\leq s \\ j+l=t}} \binom{r}{j}\binom{s}{l}(-1)^{s-l} \\
+O((t!)^2H^{t-1}d^{r+s}\abs{\CI} + 2^t t!d^{2(r+s)}H^{r+s}\sqrt{p}\log^{r+s+1}p)
\end{multline*}
when $r+s=2t$ is even, and when $r+s$ is odd, we have
\begin{equation*}
M(r,s) = O(d^{2(r+s)}H^{r+s}\sqrt{p}\log^{r+s+1}p)
\end{equation*}
if $\psi$ is nontrivial, and
\begin{equation*}
M(r,s) = O(\left(\frac{r+s-1}{2}\right)!d^{r+s-1}H^{(\frac{r+s-1}{2})}\abs{\CI} + d^{2(r+s)}H^{r+s}\sqrt{p}\log^{r+s+1}{p})
\end{equation*}
if $\psi$ is trivial.
\end{theorem}

\begin{theorem}\label{thm42}
Let $S(x)$ be real and satisfies the assumptions in Theorem \ref{thm2}. Then we have
\begin{equation*}
M(k)=O(H^{k}d^{2k}\sqrt{p}\log^{k+1}{p}).
\end{equation*}
when $k$ is odd, and
\begin{equation*}
\frac{k!}{2^{\frac{k}{2}}\left(\frac{k}{2}\right)!}\abs{\CI}H^{\frac{k}{2}}(\beta-\alpha)^{\frac{k}{2}} +O((k/2)!d^{k}H^{k/2-1}\abs{\CI} + d^{2k}H^{k}\sqrt{p}\log^{k+1}{p})
\end{equation*}
when $k$ is even.
\end{theorem}
Note that these results on the moments are slightly different from those in \cite{MaZa11,Lam13} since we did not normalize the moments here. To start the computation of $M(r,s)$, we have
\begin{align}
M(r,s) &= \sum_{x\in\CI}(\text{Re}S(x))^r(\text{Im}S(x))^s \nonumber \\
&= \sum_{x\in\CI}\left( \frac{S(x)+\overline{S(x)}}{2} \right)^r \left( \frac{S(x)-\overline{S(x)}}{2i} \right)^s \nonumber \\
&= \frac{1}{2^{r+s}i^s}\sum_{x\in\CI}\sum_{j=0}^r \binom{r}{j}S(x)^j\overline{S(x)}^{r-j} \sum_{l=0}^s (-1)^{s-l}\binom{s}{l}S(x)^l\overline{S(x)}^{s-l} \nonumber \\
&= \frac{1}{2^{r+s}i^s}\sum_{j=0}^r\sum_{l=0}^s \binom{r}{j}\binom{s}{l}(-1)^{s-l}S(j+l,r+s-(j+l)), \label{eqnMrs1}
\end{align}
where
\begin{equation*}
S(j_1,j_2)=\sum_{x\in\CI}S(x)^{j_1}\overline{S(x)}^{j_2}.
\end{equation*}
In Section \ref{sec6}, we will prove the following about $S(j_1,j_2)$.
\begin{lemma}\label{lemSj12}
Let $S(x)$ be complex and satisfies the assumptions in Theorem \ref{thm1}. If $\psi$ is non-trivial, we have
\begin{equation}\label{eqnlemSa}
S(j,j) = j!H^j\abs{\CI}(\beta-\alpha)^j+O((j!)^2H^{j-1}d^{2j}\abs{\CI} + 2^j j!H^{2j}d^{4j}\sqrt{p}\log^{2j+1}{p})
\end{equation}
when $j_1=j_2=j$, and
\begin{equation}\label{eqnlemSb}
S(j_1,j_2) = O(d^{2(j_1+j_2)}H^{j_1+j_2}\sqrt{p}\log^{j_1+j_2+1}p).
\end{equation}
when $j_1\neq j_2$.

If $\psi$ is trivial, let $a>2$ be the order of $\chi$, then
\begin{equation}\label{eqnlemSc}
S(j,j) = j!H^j\abs{\CI}(\beta-\alpha)^j+O(j!d^{2j}H^{j-1}\abs{\CI} + d^{4j}H^{2j}\sqrt{p}\log^{j+1}{p})
\end{equation}
when $j_1=j_2=j$, 
\begin{multline}\label{eqnlemSd}
S(j_1,j_2) = \\
O(\left[\frac{j_1+j_2-1}{2}\right]!d^{j_1+j_2-1}H^{[\frac{j_1+j_2-1}{2}]}\abs{\CI} + d^{2(j_1+j_2)}H^{j_1+j_2}\sqrt{p}\log^{j_1+j_2+1}{p})
\end{multline}
when $j_1-j_2$ is a multiple of $a$, and
\begin{equation}\label{eqnlemSe}
S(j_1,j_2) = O(H^{j_1+j_2}d^{2(j_1+j_2)}\sqrt{p}\log^{j_1+j_2+1}p)
\end{equation}
otherwise.
\end{lemma}
\begin{lemma}\label{lemSj12real}
Let $S(x)$ be real and satisfies the assumptions in Theorem \ref{thm2}. We have
\begin{multline}\label{eqnlemSra}
S(j_1,j_2) = \frac{(j_1+j_2)!}{2^{\frac{j_1+j_2}{2}}\left(\frac{j_1+j_2}{2}\right)!}H^{\frac{j_1+j_2}{2}}\abs{\CI}(\beta-\alpha)^{\frac{j_1+j_2}{2}} \\
+O(j!d^{j_1+j_2}H^{\frac{j_1+j_2}{2}-1}\abs{\CI} + d^{2(j_1+j_2)}H^{j_1+j_2}\sqrt{p}\log^{j_1+j_2+1}{p})
\end{multline}
when $j_1+j_2$ is even, and
\begin{equation}\label{eqnlemSrb}
S(j_1,j_2) = O(H^{j_1+j_2}d^{2(j_1+j_2)}\sqrt{p}\log^{j_1+j_2+1}p)
\end{equation}
when $j_1+j_2$ is odd.
\end{lemma}

Expanding the sum $S(x)$ in $S(j_1,j_2)$, we obtain
\begin{align}
S(j_1,j_2)&=\sum_{x\in\CI}\sum_{\substack{P_{1}\in C, x<x(P_{1})\leq x+H \\ y(P_{1})\in\mathcal{J}}}\cdots\sum_{\substack{P_{j_1+j_2}\in C, x<x(P_{j_1+j_2})\leq x+H \\ y(P_{j_1+j_2})\in\mathcal{J}}} \label{eqnsumrange}\\ 
&\qquad \prod_{l=1}^{j_1}\chi(g(P_{l}))\psi(f(P_{l}))\prod_{l=j_1+1}^{j_1+j_2}\bar{\chi}(g(P_{l}))\bar{\psi}(f(P_{l})) \nonumber \\
&=\sum_{x\in\CI}\sum_{\substack{P_{1}\in C, x<x(P_{1})\leq x+H \\ y(P_{1})\in\mathcal{J}}}\cdots\sum_{\substack{P_{j_1+j_2}\in C, x<x(P_{j_1+j_2})\leq x+H \\ y(P_{j_1+j_2})\in\mathcal{J}}} \nonumber \\
&\qquad \chi\left( \frac{g(P_{1})\ldots g(P_{j_1})}{g(P_{j_1+1})\ldots g(P_{j_1+j_2})} \right)\psi\left( \sum_{l=1}^{j_1}f(P_{l})-\sum_{l=j_1+1}^{j_1+j_2}f(P_{l}) \right). \nonumber
\end{align}

When $C$ is the affine line and $\CJ$ is the full interval, all the $x(P_i)$ lie on a straight line and the usual way is to proceed by switching the order of summation and then use the classical Bombieri-Weil bound on $\mathbb{A}^1$. For a general curve the contents inside the characters are generally not a rational function on $C$. So the usual strategy does not work. Instead, we will introduce two families of auxiliary curves $C_{\mathbf{h}}$ and $C^o_{\mathbf{h}}$ in the next section, and transform the contents inside the characters into rational functions on these curves.

\section{The auxillary curve $C_{\mathbf{h}}$}\label{sec5}

Recall that $C\subseteq\mathbb{A}^2_p:=\mathbb{A}^2(\mathbb{F}_p)$ is an absolutely irreducible affine plane curve (not necessarily smooth) over $\mathbb{F}_p$ of degree $d>1$, defined by the equation $P(x,y)=0$. \textit{Fix} $r$ integers $h_1,\ldots,h_r$, which may or may not be distinct. Let $\mathbf{h}=(h_1,\ldots,h_r)$. Define the variety $C_{\mathbf{h}}\subseteq\mathbb{A}^{r+1}$ by the following system of equations:
\begin{equation} \label{defCH}
P(x+h_i,y_i)=0 \,\,\forall \,\, 1\leq i\leq r.
\end{equation}
Similar constructions have appeared in \cite{MaZa11, MaZa12, Mak14}. Note that there are totally $r$ equations and $r+1$ variables ($x$ and $y_i$ for $1\leq i\leq r$). It is easy to see that $C_{\mathbf{h}}$ is a curve for any $\mathbf{h}$. Its degree is at most $d^{r}$, where $d$ is the degree of $C$.

A point $(x,y_1,\ldots,y_r)$ corresponds to an $r$-tuple $(P_1,\ldots,P_r)$ of points on $C$ such that $x(P_i)=x+h_i$ for all $1\leq i\leq r$. It is clear that this correspondence is one-to-one. Therefore, the sum in \eqref{eqnsumrange} is the same as summing points on $C_{\mathbf{h}}$ for all possible $\mathbf{h}=(h_1,\ldots,h_r)$ with $0 h_i\leq H$.

We are interested in the irreducibility of this curve $C_{\mathbf{h}}$. It is not difficult to see that for any $C$, the curve $C_{\mathbf{h}}$ cannot be irreducible if there are $i\neq j$ such that $h_i=h_j$. The total number of $\mathbf{h}$ with this property is at most $O(H^{r-1})$. The following proposition states that most of the other $C_{\mathbf{h}}$ are irreducible.
\begin{prop}\label{prop51}
Let $p$ be a large prime, and let $C$ be an absolutely irreducible curve over $\mathbb{F}_p$. Let $H=H(p)>0$ be an integer function that tends to infinity as $p$ tends to infinity, and let $\mathbf{h}=(h_1,\ldots,h_r)$ with $0< h_i\leq H$ for all $i$. The curve $C_{\mathbf{h}}$ is absolutely irreducible except for at most $O_r(H^{r-1})$ of them.
\end{prop}

To prove Proposition \ref{prop51}, we first construct a variety related to $C_{\mathbf{h}}$. Let $V$ be the variety defined by the set of equations \eqref{defCH}, but with the $h_i$ also considered as variables. Thus $V$ is of dimension $r$. The structure of $V$ is very simple.
\begin{lemma}
The variety $V$ is isomorphic to $C^r$. In particular $V$ is absolutely irreducible (since $C$ is).
\end{lemma}
\begin{proof}
Let $C^r$ be defined by $P(x_i,y_i)=0$ for $1\leq i\leq r$. The map $C^r\rightarrow V$ given by $x_i\mapsto x+h_i$ (and $y_i$ maps to itself) is clearly a (linear) isomorphism.
\end{proof}

Let $\phi: C^r\rightarrow \mathbb{A}^{r-1}$ be the fibration given by
\begin{equation}\label{defphi}
\phi(x_1,\ldots,x_r,y_1,\ldots,y_r)=(x_2-x_1,x_3-x_1,\ldots,x_r-x_1).
\end{equation}
It is easy to see that $C_{\mathbf{h}}$ is isomorphic to $C_{\mathbf{h}+a}$ for any $a$, where $\mathbf{h}+a=(h_1+a,\ldots,h_r+a)$. Therefore, it suffices show that the fibres of $\phi$ over the box $(-H,H]^{r-1}$ are absolutely irreducible except for at most $O_r(H^{r-2})$ of them, because $H$ translate of the tuples $(0,h_2,\ldots,h_r)$ with $-H< h_i\leq H$ covers $(0,H]^{r}$. Since the number of fibres with at least one pair of $x_i=x_j$ is $O_r(H^{r-2})$, we may assume that all the $x_i$'s are distinct.

Let $\CC$ be the projectivization of $C$, and consider the rational map $\phi:\CC^r\dashedrightarrow \mathbb{P}^{r-1}$ extending the map $\phi$ in \eqref{defphi}. Let $\tilde{\CC}$ be the normalization of $\CC$. Let $\tilde{\phi}:\tilde{\CC}^r\dashedrightarrow \mathbb{P}^{r-1}$ be the map that makes the following diagram commutes.
\begin{equation*}
\xymatrix{
\tilde{\CC}^r \ar@{->}[d] \ar@{-->}[dr]^{\tilde{\phi}} & \\
\CC^r \ar@{-->}[r]^{\phi} & \mathbb{P}^{r-1}
}
\end{equation*}
The generic fibre of $\tilde{\phi}$ is irreducible since $\tilde{\CC}^r$ is (the generic point of $\tilde{\CC}^r$ maps to the generic point of $\mathbb{P}^{r-1}$). Since $\mathbb{P}^{r-1}$ is normal, every fibre of $\tilde{\phi}$ is connected by Zariski's connectedness theorem \cite{Gro61}. 

On the other hand, as the $x_i$'s are all distinct, a point $P=(h_2,\ldots,h_r)\in\mathbb{P}^{r-1}$ is a regular point of $\tilde{\phi}$ unless there are distinct $1\leq i,j\leq r$ such that both $x_i$ and $x_j$ are critical values of the map $\tilde{\CC}\rightarrow\mathbb{P}^1_x$ (here $\mathbb{P}^1_x$ is the projective line with variable $x$). These are all \textit{linear} conditions. Therefore, a point $P$ can be a critical value only if it lies on a union of proper linear subspaces in $\mathbb{P}^{r-1}$. Since the number of points inside a linear subspace $Q$ whose all coordinates are in $(-H,H]$ is at most $O_r(H^{r-2})$, the number of critical values of $\tilde{\phi}$ is at most $O_r(H^{r-2})$. The fibre over a regular point is smooth, so the fibres of $\tilde{\phi}$ are smooth except for at most $O_r(H^{r-2})$ of them.

We have shown that except for at most $O_r(H^{r-2})$ of the fibres, all other fibres of $\tilde{\phi}$ are both connected and smooth, hence they are absolutely irreducible. Their corresponding fibres in $\phi:\CC^r\rightarrow\mathbb{P}^{r-1}$ are thus absolutely irreducible. This completes the proof of Proposition \ref{prop51}.

\begin{remark}
It is easy to see that for any $\mathbf{h}$, the curve $C_{\mathbf{h}}$ is rational when $C$ is rational. Hence they are irreducible. When $C$ is hyperelliptic, and when all the $h_i$ are distinct, then the author \cite{Mak14} showed that $C_{\mathbf{h}}$ is irreducible. It is very possible that the same is true for all curve $C$, i.e. the set
\begin{equation*}
\{ (h_1,\ldots,h_r)\in (0,H]^r: h_i\neq h_j \text{~for~} i\neq j, C_{\CH} \text{~is not absolutely irreducible} \}
\end{equation*}
is empty. In any case, a more precise estimate on the cardinality of the above set will improve the error terms of the moments, but the author was not able to prove anything better than Proposition \ref{prop51}.
\end{remark}

Next, we let $C^o_{\mathbf{h}}$ be defined by the equations \eqref{defCH} together with the condition that whenever $h_i=h_j$ for some $i\neq j$, then we require $y_i\neq y_j$. Thus a point $(x,y_1,\ldots,y_r)$ corresponds to an $r$-tuple $(P_i,\ldots,P_r)$ such that $x(P_i)=x+h_i$ and all $P_i$ are distinct. Note that $C^o_{\mathbf{h}}=C_{\mathbf{h}}$ when all the $h_i$ are distinct. It is easy to see that $C^o_{\mathbf{h}}$ is obtained from $C_{\mathbf{h}}$ by removing finitely many points, so it is an open set in the curve $C_{\mathbf{h}}$. In particular, $C^o_{\mathbf{h}}$ is itself an affine curve (see \cite[Lemma I.4.2]{Har77}). As an open affine in an absolutely irreducible curve, we immediately have the following for $C^o_{\mathbf{h}}$.
\begin{cor}\label{cor52}
Let $\mathbf{h}=(h_1,\ldots,h_r)$ with $0< h_i\leq H$ for all $i$. The curve $C^o_{\mathbf{h}}$ is absolutely irreducible except for at most $O_r(H^{r-1})$ of them.
\end{cor}

Next, we will consider a family of rational functions on $C_{\mathbf{h}}$, and determine whether it is non-trivial in the sense that Lemma \ref{lem:perel} is applicable to such functions. We first describe our setting. Let $\mathbf{h}=(h_1,\ldots,h_{j_1+j_2})$ be as usual. Let $y_1,\ldots,y_{j_1+j_2}$ be a set of indeterminates that may or may not be distinct, but we impose the restriction that $y_i$ and $y_j$ can stand for the same indeterminate only if $h_i=h_j$ (the converse need not hold). Let $\CY$ be the set of all indeterminates, i.e. $\CY=\{y_1,\ldots,y_{j_1+j_2}\}$ (with multiplicities discarded). Clearly $\CY$ depends on $\mathbf{h}$, but one $\mathbf{h}$ can yield different $\CY$ according to whether we give two distinct variables to a pair $h_i=h_j$ or not. We will say $\CY$ belongs to $\mathbf{h}$ if a set of indeterminates $\CY=\{y_1,\ldots,y_{j_1+j_2}\}$ can be constructed from $\mathbf{h}$ in the above fashion. Let $I_{\CY}$ be the set of indices corresponding to the distinct indeterminates in $\CY$, i.e. $\CY=\{y_i: i\in I_{\CY}\}$, and let 
\begin{equation*}
\mathbf{t}=\mathbf{t}(\CY):=\{h_i: i\in I_{\CY}\},
\end{equation*}
where as an ordered pair we require it to preserve the original order of $\mathbf{h}$ and keep only the first occurrence of duplicated $y_i$.

Let $f\in\mathbb{F}_p(x,y)$ be a rational function on $C$, and consider the combination
\begin{equation}\label{eqnFY}
F(x,\CY)=\sum_{j=1}^{j_1}f(x+h_j,y_j)-\sum_{j=j_1+1}^{j_1+j_2}f(x+h_j,y_j).
\end{equation}
Then $F$ can be viewed as a rational function on the curve $C_{\mathbf{t}}$, and if we want distinct indeterminates to correspond to distinct points on $C$, we can view $F$ as a rational function on $C^o_{\mathbf{t}}$. The following proposition characterizes when is $F(x,\CY)$ non-trivial.
\begin{prop}\label{prop53}
Let $p$ be a large prime, and let $C$ be an absolutely irreducible plane curve over $\mathbb{F}_p$. Let $f\in\mathbb{F}_p(x,y)$ be a rational function on $C$, $f=f_1/f_2$, $f_1,f_2\in\mathbb{F}_p[x,y]$, $\deg{f_1},\deg{f_2}<p$ and $f_1,f_2$ has no common factors. Suppose that $f$ is not linear on $C$, and subject to the following conditions:
\begin{enumerate}
\item If $f$ is a polynomial, then write $f(x,y)=r_1(x)+r_2(x,y)$, where $r_1$ consists of all terms which do not depend on $y$. We further assume that either $r_2$ is not linear, or if $r_2$ is linear, then $\deg{r_1}\geq 3$.
\item If $f$ is not a polynomial, i.e. $\deg{f_2}\geq 1$, then assume $\deg{f_2}=o(\log{p})$ is small.
\end{enumerate}
Let $H$, $j_1,j_2$ be positive integers so that both $H=o(\log{p})$ and $j_1+j_2=o(\log{p})$. Let $0< h_1,\ldots, h_{j_1+j_2}\leq H$ be integers, which may or may not be distinct. Let $F(x,\CY)$ and $\mathbf{t}(\CY)$ be defined as in \eqref{eqnFY}. If $F=\tilde{h}^p-\tilde{h}+\text{(linear terms)}$ for some rational function $\tilde{h}$ on any irreducible component of $C_{\mathbf{t}(\CY)}$, then we must have $j_1=j_2$ and $F(x,\CY)$ is the zero polynomial.
\end{prop}
Note that since $C^o_{\mathbf{t}}$ differs from $C_{\mathbf{t}}$ only be omitting finitely many points, the same conclusion holds (under the same assumption) for $C^o_{\mathbf{t}}$.
\begin{proof}
Since we are considering functions modulo $C$, we may assume that the $y$-degree of $f_1$ and $f_2$ are both smaller than the $y$-degree of $C$. Collect the terms in $F$ that correspond to the same indeterminate $y_i$ and renaming if necessary, we get
\begin{equation}\label{eqnprop53a}
F=m_1f(x+u_1,y_1)+\ldots+m_rf(x+u_r,y_r),
\end{equation}
where the $m_i$ are integers, $0<u_i\leq H$ may or may not be distinct, and the $y_i$ are distinct indeterminates. It suffices to show that $m_1,\ldots,m_r$ are all zero.

Assume on the contrary that not all $m_i$ are zero, then by removing the $m_i$ that are zero, we may assume that $m_i\neq 0$ for all $i$ in \eqref{eqnprop53a}. Let $\Gamma$ be a component of $C_{\mathbf{t}}$ such that $F$ is of the form $F=\tilde{h}^p-\tilde{h}+\text{(linear terms)}$ on $\Gamma$. Note that $\Gamma$ is a curve.

First suppose $f$ is a polynomial. Then $F$ is also a polynomial. Since $\deg{f}<p$, we also have $\deg{F}<p$. The special form of $F$ implies that $F$ is linear on $\Gamma$. Let $C_i$ be the curve defined by $P(x+u_i,y_i)=0$. Consider the $i$-th projection
\begin{equation*}
\pi_i: \Gamma \subseteq C_{\mathbf{t}} \longrightarrow C_i.
\end{equation*}
Since $C_i$ are isomorphic to $C$, they are absolutely irreducible. Thus at least one of the $\pi_i$ are surjective. Write $f(x,y)=r_1(x)+r_2(x,y)$, then
\begin{equation*}
F(x,y)=(m_1r_1(x+u_1)+\ldots+m_rr_1(x+u_r))+(m_1r_2(x,y_1)+\ldots+m_rr_2(x,y_r)).
\end{equation*}
The surjectivity of $\pi_i$ means that $F$ must be linear on $y_i$. 
By our construction on $r_2$ this implies $r_2(x,y)=c_2y$ is linear. Hence,
\begin{equation*}
R_1(x)=m_1r_1(x+u_1)+\ldots+m_rr_1(x+u_r)
\end{equation*}
is linear on $\Gamma$. Thus $R_1(x)$ itself is linear since $R_1$ depends only on $x$ but not the $y_i$'s. This is not possible since $\deg{r_1}\geq 3$ and $j_1+j_2=o(\log{p})$, which means the coefficient of $x^{\deg{r_1}-1}$ in $R_1(x)$ does not vanish.

On the other hand, suppose $\deg{f_2}\geq 1$. From \eqref{eqnprop53a}, we have
\begin{equation*}
F=\frac{F_1}{f_2(x+u_1,y_1)\ldots f_2(x+u_r,y_r)}
\end{equation*}
for some polynomial $F_1(x,y_1,\ldots,y_r)$. The assumptions that both $j_1+j_2$ and $\deg{f_2}$ are small implies that the denominator has degree less than $p$. Thus the denominator of $F$ is nonconstant and has degree less than $p$ on any irreducible component. So $F$ cannot be in the form $F=\tilde{h}^p-\tilde{h}+\text{(linear terms)}$ neither.
\end{proof}

\section{Computation of $S(j_1,j_2)$}\label{sec6}

In this section we will continue our calculation of the moments. The sum in \eqref{eqnsumrange} is the same as summing over the points on the curves $C_{\mathbf{h}}$ for all $\mathbf{h}=(h_1,\ldots,h_{j_1+j_2})$ with $0< h_i \leq H$, with the correspondence $P_i\leftrightarrow(x+h_i,y_i)$, i.e. 
\begin{equation}\label{eqnsplit1}
\sum_{x\in\CI}\sum_{\substack{P_{1}\in C, x<x(P_{1})\leq x+H \\ y(P_{1})\in\mathcal{J}}}\cdots\sum_{\substack{P_{j_1+j_2}\in C, x<x(P_{j_1+j_2})\leq x+H \\ y(P_{j_1+j_2})\in\mathcal{J}}} = \sum_{\mathbf{h}\in(0,H]^{j_1+j_2}}\sum_{\substack{(x,y_1,\ldots,y_{j_1+j_2})\in C_{\mathbf{h}} \\ x\in\CI, y_i\in\CJ}}
\end{equation}
However, we will need a finer splitting of the sum. Let $P_i=(x+h_i,y_i)$. We regard each tuple $(P_1,\ldots,P_{j_1+j_2})$ as a point on the curve
\begin{equation*}
P(x+h_i,y_i)=0 \,\,\forall \,\, 1\leq i\leq j_1+j_2,
\end{equation*}
but with $y_i,y_j$ stand for the same indeterminate if and only if $P_i=P_j$ (thus, in contrary to \eqref{defCH}, there may be duplicated equations above). In this way the tuple is viewed as a point on $C^o_{\mathbf{t}(\CY)}$, where $\mathbf{t}(\CY)$ is defined as in Section \ref{sec5}, and the contents inside the characters of \eqref{eqnsumrange} become rational functions on $C^o_{\mathbf{t}(\CY)}$. Each such tuple $(P_1,\ldots,P_{j_1+j_2})$ lies on exactly one of these curves $C^o_{\mathbf{t}(\CY)}$ for an $\mathbf{h}=(h_1,\ldots,h_{j_1+j_2})$ with $0<h_1\leq H$, and $\CY$ belongs to $\mathbf{h}$. Let 
\begin{equation*}
Y_{\mathbf{h}}=\{\CY: \CY \text{~belongs to~} \mathbf{h} \}.
\end{equation*}
Then the sum \eqref{eqnsumrange} can be split as
\begin{align}
S(j_1,j_2) &= \sum_{\mathbf{h}\in(0,H]^{j_1+j_2}}\sum_{\CY\in Y_{\mathbf{h}}}\sum_{\substack{x\in\CI,y_i\in\CJ \\ \mathbf{x}\in C^o_{\mathbf{t}(\CY)}}}
\chi\left( \frac{g(x+h_1,y_1)\ldots g(x+h_{j_1},y_{j_1})}{g(x+h_{j_1+1},y_{j_1+1})\ldots g(x+h_{j_1+j_2},y_{j_1+j_2})} \right) \nonumber \\ 
&\qquad \psi\left( \sum_{l=1}^{j_1}f(x+h_l,y_l)-\sum_{l=j_1+1}^{j_1+j_2}f(x+h_l,y_l) \right) \nonumber \\
&= \sum_{\mathbf{h}\in(0,H]^{j_1+j_2}}\sum_{\CY\in Y_{\mathbf{h}}}\sum_{\substack{x\in\CI,y_i\in\CJ \\ \mathbf{x}\in C^o_{\mathbf{t}(\CY)}}} \chi(G(x,\CY))\psi(F(x,\CY)), \label{eqnsplit2}
\end{align}
where
\begin{equation}\label{eqnGY}
G(x,\CY)=\frac{g(x+h_1,y_1)\ldots g(x+h_{j_1},y_{j_1})}{g(x+h_{j_1+1},y_{j_1+1})\ldots g(x+h_{j_1+j_2},y_{j_1+j_2})}
\end{equation}
and $F(x,\CY)$ is defined in \eqref{eqnFY}. Note that when $h_i$ are all distinct, the splitting \eqref{eqnsplit2} above is the same as \eqref{eqnsplit1}. To evaluate $S(j_1,j_2)$, we will need to consider two cases depending on whether $\psi$ is trivial or not.

\subsection{The case when $\psi$ is non-trivial}

In this subsection we assume that $\psi$ is non-trivial. We recall from \eqref{eqnFY} that
\begin{equation*}
F(x,\CY)=\sum_{j=1}^{j_1}f(x+h_j,y_j)-\sum_{j=j_1+1}^{j_1+j_2}f(x+h_j,y_j).
\end{equation*}
By Proposition \ref{prop53} (and our assumptions on $f$), the function $F(x,\CY)$ is nonlinear on $C^o_{\mathbf{t}(\CY)}$ unless $j_1=j_2=j$ and $F(x,\CY)$ is zero. When $f$ depends on $y$, this implies $(y_{j+1},\ldots,y_{2j})$ is a permutation of $(y_1,\ldots,y_j)$, and $G(x,\CY)$ is then automatically zero. When $f$ does not depend on $y$, $F(x,\CY)=0$ if $(h_{j+1},\ldots,h_{2j})$ is a permutation of $(h_1,\ldots,h_j)$, but then we will still need $(y_{j+1},\ldots,y_{2j})$ to be a permutation of $(y_1,\ldots,y_j)$ in order to make $G(x,\CY)$ zero unless $g$ does not depend on $y$. Thus if any of the $f$ or $g$ depends on $y$, Lemma \ref{lem:perel} is applicable unless $(y_{j+1},\ldots,y_{2j})$ to be a permutation of $(y_1,\ldots,y_j)$. If both $f$ and $g$ does not depend on $y$, then the sum is equivalent to one that has trivial $\psi$ and $g$ does not depend on $y$. Such sums will be treated in the next subsection.

Let $\CE$ be the set of $\mathbf{t}(\CY)$ that contribute to the diagonal terms, then the number of $\mathbf{t}(\CY)$ that are off-diagonal is $O(H^{j_1+j_2})$. Therefore, by Lemma \ref{lem:perel},
\begin{multline}\label{eqnoffdiag}
\sum_{\mathbf{h}\in(0,H]^{j_1+j_2}}\sum_{\substack{\CY\in Y_{\mathbf{h}} \\ \mathbf{t}(\CY)\notin\CE}}\sum_{\substack{x\in\CI,y_i\in\CJ \\ \mathbf{x}\in C^o_{\mathbf{t}(\CY)}}}
\chi(G(x,\CY))\psi(F(x,\CY)) \\
= O(d^{2(j_1+j_2)}H^{j_1+j_2}\sqrt{p}\log^{j_1+j_2+1}p).
\end{multline}
When $j_1\neq j_2$, all terms are non-diagonal and the above gives \eqref{eqnlemSb} in Lemma \ref{lemSj12}. 

Now suppose that $j_1=j_2=j$ and we have diagonal terms. In this case we have $F(x,\CY)=0$, and this automatically implies $G(x,\CY)=0$. Thus the contribution of each $\mathbf{t}(\CY)\in\CE$ to the sum is exactly the number of points on $C^o_{\mathbf{t}(\CY)}$ inside the box $\CB:=\CI\times [\alpha p,\beta p)$. We now count the number of such $\mathbf{t}(\CY)$. When $y_1,\ldots,y_j$ are all distinct (and $y_{j+1},\ldots,y_{2j}$ is a permutation of the $y_1,\ldots,y_j$), then $\mathbf{t}(\CY)=(h_1,\ldots,h_j)=\mathbf{h}$ and $C^o_{\mathbf{t}(\CY)}=C^o_{\mathbf{h}}=C_{\mathbf{h}}$, and there are a total of $H^j+O(j^2H^{j-1})$ such $\mathbf{h}$. Corollary \ref{cor52} shows that all but $O(H^{j-1})$ of the $C^o_{\mathbf{h}}$ are absolutely irreducible, and since the $\mathbf{F}_p$-points are uniformly distributed on an affine curve (see \cite[Corollary 2.7]{MaZa12}), each of these irreducible curves contribute
\begin{equation}\label{eqnNBirr}
N_{\CB}(C^o_{\mathbf{h}}) = \abs{\CI}(\beta-\alpha)^j+O(d^{2j}\sqrt{p}\log^{j+1}{p})
\end{equation}
to the main term. For those $C_{\mathbf{h}}$ that are not absolutely irreducible, the maximum possible number of components is bounded by the degree, which is $O(d^{2j})$. Thus
\begin{equation}\label{eqnNBnonirr}
N_{\CB}(C^o_{\mathbf{h}}) = O(d^{2j}\abs{\CI}).
\end{equation}
For any $(y_1,\ldots,y_j)$ with distinct components, there are $j!$ possible permutations of $(y_{j+1},\ldots,y_{2j})$, therefore these terms yield a total contribution of
\begin{align}\label{eqn61}
& j!(H^j+O(j^2H^{j-1}))(\abs{\CI}(\beta-\alpha)^j+O(d^{2j}\sqrt{p}\log^{j+1}{p}))+j!O(d^{2j}H^{j-1}\abs{\CI}) \nonumber \\
=~& j!H^j\abs{\CI}(\beta-\alpha)^j+O(j!j^2H^{j-1}d^{2j}\abs{\CI} + j!H^jd^{2j}\sqrt{p}\log^{j+1}{p})
\end{align}
to the main term. The other terms on the diagonal all have $y_{i_1}=y_{i_2}$ for some $i_1\neq i_2$ and $1\leq i_1,i_2\leq j$. The total number of such $t(\CY)$ is at most $O(j!H^{j-1})$, each such curve $C^o_{\mathbf{t}(\CY)}$ has at most $O(d^{2j}\abs{\CI})$ points, and there are at most $j!$ permutations of $(y_{j+1},\ldots,y_{2j})$. So these terms together give a contribution of
\begin{equation}\label{eqn62}
O((j!)^2H^{j-1}d^{2j}\abs{\CI})
\end{equation}
to the main term, which is small compared to \eqref{eqn61}. Combining \eqref{eqn61} and \eqref{eqn62}, we get the total contribution of the diagonal, which is
\begin{multline}\label{eqnondiag}
\sum_{\mathbf{h}\in(0,H]^{j_1+j_2}}\sum_{\substack{\CY\in Y_{\mathbf{h}} \\ \mathbf{t}(\CY)\in\CE}}\sum_{\substack{x\in\CI,y_i\in\CJ \\ \mathbf{x}\in C^o_{\mathbf{t}(\CY)}}}
\chi(G(x,\CY))\psi(F(x,\CY)) \\ 
= j!H^j\abs{\CI}(\beta-\alpha)^j+O((j!)^2H^{j-1}d^{2j}\abs{\CI} + j!H^jd^{2j}\sqrt{p}\log^{j+1}{p}).
\end{multline}
Finally, combining \eqref{eqnoffdiag} and \eqref{eqnondiag}, we obtain
\begin{equation*}
S(j,j)=j!H^j\abs{\CI}(\beta-\alpha)^j+O((j!)^2H^{j-1}d^{2j}\abs{\CI} + j!H^{2j}d^{4j}\sqrt{p}\log^{2j+1}{p}).
\end{equation*}
This gives \eqref{eqnlemSa} in Lemma \ref{lemSj12}.

\subsection{The case when $\psi$ is trivial}

Next we deal with the case when $\psi$ is trivial. Using the same splitting as in the previous case, \eqref{eqnsumrange} becomes
\begin{equation}\label{eqnpsi1}
S(j_1,j_2)=\sum_{\mathbf{h}\in(0,H]^{j_1+j_2}}\sum_{\CY\in Y_{\mathbf{h}}}\sum_{\substack{x\in\CI,y_i\in\CJ \\ \mathbf{x}\in C^o_{\mathbf{t}(\CY)}}} \chi(G(x,\CY)),
\end{equation}
where $G(x,\CY)$ is defined in \eqref{eqnGY}, i.e.
\begin{equation*}
G(x,\CY)=\frac{g(x+h_1,y_1)\ldots g(x+h_{j_1},y_{j_1})}{g(x+h_{j_1+1},y_{j_1+1})\ldots g(x+h_{j_1+j_2},y_{j_1+j_2})}.
\end{equation*}
Let $a$ be the order of $\chi$. Lemma \ref{lem:perel} is applicable when $G(x,\CY)$ is not a complete $a$-th power on any irreducible component of $C^o_{\mathbf{t}(\CY)}$. We need to consider two cases according to whether $g$ depends on $y$ or not.

First, if $g$ depends on $y$, then products and quotients of distinct $y_i$'s cannot be a complete $a$-th power. Hence, if the $g(x+h_i,y_i)$ stack up and become a complete $a$-th power, it must come from $a$ terms with the same $y_i$ (hence $h_i$), or that the same amount of such terms appear in both the numerator and the denominator. In this case, we may obtain a complete $a$-th power by having clusters that have the same $y_i$ in the numerator and the denominator of $G$, and group the remaining terms into clusters each consisting of $a$ terms with the same $y_i$. We will call the terms that can completely form clusters in $G$ the diagonal terms. Note that for $a=2$, both types of clusters need exactly two terms, and this will be the reason that the case $a=2$ exhibits a different behaviour. For the non-diagonal terms, we can apply Lemma \ref{lem:perel} to obtain the estimate
\begin{equation}\label{eqnpsioff}
\sum_{\substack{x\in\CI,y_i\in\CJ \\ \mathbf{x}\in C^o_{\mathbf{t}(\CY)}}} \chi(G(x,\CY))=O(d^{2(j_1+j_2)}\sqrt{p}\log^{j_1+j_2+1}p).
\end{equation}

If $j_1-j_2$ is not a multiple of $a$, then $G(x,\CY)$ cannot be a complete $a$-th power and there are no diagonal terms. All terms in \eqref{eqnpsi1} can be estimated using \eqref{eqnpsioff}. Thus
\begin{equation*}
S(j_1,j_2) = O(H^{j_1+j_2}d^{2(j_1+j_2)}\sqrt{p}\log^{j_1+j_2+1}p).
\end{equation*}
This is \eqref{eqnlemSe} of Lemma \ref{lemSj12}.

Now suppose $j_1-j_2=ma$ for some integer $m$. Let $j=\min\{j_1,j_2\}$. There are $j$ pairs of $g(x+h_i,y_i)$ that have the same $y_i$ in the numerator and the denominator of $G$, and the remaining terms are in $\abs{m}$ clusters each consisting of $a$ terms with the same $y_i$. For such $j_1$ and $j_2$, it is not difficult to count the total number of such $\mathbf{t}(\CY)$ that make $G$ a complete $a$-th power, which is
\begin{equation}\label{eqn64}
j!\frac{(\abs{m}a)!}{(a!)^{\abs{m}}\abs{m}!}H^{j+\abs{m}}(1+O(j^2/H))
\end{equation}
when $a>2$, and is
\begin{equation*}
\frac{(j_1+j_2)!}{2^{\frac{j_1+j_2}{2}}\left(\frac{j_1+j_2}{2}\right)!}H^{(j_1+j_2)/2}(1+O((j_1+j_2)^2/H))
\end{equation*}
when $a=2$. 

For $a>2$, $S(j_1,j_2)$ is the largest when $m=0$ and $j_1=j_2=j$. In this case, there are a total of $j!H^j(1+O(j^2/H))$ diagonal terms, and the corresponding $C^o_{\mathbf{t}(\CY)}$ are absolutely irreducible except for at most $H^{j-1}$ of them. Using \eqref{eqnNBirr} for irreducible ones, and \eqref{eqnNBnonirr} for reducible ones, we see that the contribution of such terms to the diagonal is
\begin{equation}\label{eqn63}
j!H^j\abs{\CI}(\beta-\alpha)^j+O(j!d^{2j}H^{j-1}\abs{\CI} + d^{2j}\sqrt{p}\log^{j+1}{p}).
\end{equation}
Combining \eqref{eqn63} with the error \eqref{eqnpsioff} from the off-diagonal terms (a total of $O(H^{2j})$ such terms), we get
\begin{equation*}
S(j,j) = j!H^j\abs{\CI}(\beta-\alpha)^j+O(j!d^{2j}H^{j-1}\abs{\CI} + d^{4j}H^{2j}\sqrt{p}\log^{j+1}{p}).
\end{equation*}
This proves \eqref{eqnlemSc} of Lemma \ref{lemSj12}. When $m\neq 0$, the total number of $\mathbf{t}(\CY)$ that contribute to the diagonal is at most $O([\frac{j_1+j_2-1}{2}]!H^{[\frac{j_1+j_2-1}{2}]})$, and so their contribution to the diagonal is at most $O([\frac{j_1+j_2-1}{2}]!d^{j_1+j_2-1}H^{[\frac{j_1+j_2-1}{2}]}\abs{\CI})$, which is small. Together with the non-diagonal terms, we get
\begin{multline*}
S(j_1,j_2) = \\
O(\left[\frac{j_1+j_2-1}{2}\right]!d^{j_1+j_2-1}H^{[\frac{j_1+j_2-1}{2}]}\abs{\CI} + d^{2(j_1+j_2)}H^{j_1+j_2}\sqrt{p}\log^{j_1+j_2+1}{p}),
\end{multline*}
which is \eqref{eqnlemSd} of Lemma \ref{lemSj12}.

The case for $a=2$ is slightly different. The off-diagonal terms can still be estimated by \eqref{eqnpsioff}, and when $j_1+j_2$ is odd, there are no diagonal terms. This gives \eqref{eqnlemSrb} in Lemma \ref{lemSj12real}. When $j_1+j_2$ is even, the number of $\mathbf{t}(\CY)$ that contributes to the diagonal is $\frac{(j_1+j_2)!}{2^{\frac{j_1+j_2}{2}}\left(\frac{j_1+j_2}{2}\right)!}H^{\frac{j_1+j_2}{2}}(1+O((j_1+j_2)^2/H))$, and the corresponding $C^o_\mathbf{t}(\CY)$ are absolutely irreducible except for $O(H^{\frac{j_1+j_2}{2}-1)}$ of them. Using \eqref{eqnNBirr} for irreducible ones, \eqref{eqnNBnonirr} for reducible ones, and combining this with the error from non-diagonal terms \eqref{eqnpsioff}, we obtain
\begin{multline*}
S(j_1,j_2) = \frac{(j_1+j_2)!}{2^{\frac{j_1+j_2}{2}}\left(\frac{j_1+j_2}{2}\right)!}H^{\frac{j_1+j_2}{2}}\abs{\CI}(\beta-\alpha)^{\frac{j_1+j_2}{2}} \\
+O(j!d^{j_1+j_2}H^{\frac{j_1+j_2}{2}-1}\abs{\CI} + d^{2(j_1+j_2)}H^{j_1+j_2}\sqrt{p}\log^{j_1+j_2+1}{p}),
\end{multline*}
which is \eqref{eqnlemSra}.

When $g$ does not depend on $y$, then the assumption that $\deg{g}$ is small ensures that product and quotients of distinct $h_i$'s cannot be a complete $a$-th power. Hence, if the $g(x+h_i,y_i)$ stack up and become a complete $a$-th power, all these terms must have the same $h_i$ (but contrary to the first case, here the $y_i$ need not be the same). In this case, we may obtain a complete $a$-th power by having clusters that have the same $h_i$ in the numerator and the denominator of $G$, and group the remaining terms into clusters each consisting of $a$ terms with the same $h_i$. We will use the splitting \eqref{eqnsplit1}, which depends only on $\mathbf{h}$. The calculations are very similar to the case when $g$ depends on $y$, and we have the same estimation as that case. We will calculate the case when $a>2$ and $j_1=j_2=j$, the other cases are similar and we leave them to the reader.

As in the previous case, the number of $\mathbf{h}$ that lie inside the diagonal is \eqref{eqn64} with $m=0$, for other terms we can apply Lemma \ref{lem:perel} and get
\begin{equation*}
\sum_{\substack{(x,y_1,\ldots,y_{j_1+j_2})\in C_{\mathbf{h}} \\ x\in\CI, y_i\in\CJ}} G(x,\CY) = O(d^{4j}H^{2j}\sqrt{p}\log{p}).
\end{equation*}
By Proposition \ref{prop51}, each of the $C_{\mathbf{h}}$ is absolutely irreducible except for at most $O(H^{j})$ of them. For those irreducible curves the estimate in \eqref{eqnNBirr} is applicable to $C_{\mathbf{h}}$ as well, and for those reducible curves \eqref{eqnNBnonirr} applies. Thus combining the above two equalities and the above estimate for off-diagonal terms, we have
\begin{equation*}
S(j,j) = j!H^j\abs{\CI}(\beta-\alpha)^j+O(j!d^{2j}H^{j-1}\abs{\CI} + d^{4j}H^{2j}\sqrt{p}\log{p}).
\end{equation*}

\section{Computation of the moments $M(r,s)$}

In this section we will finish our computation of the moments $M(r,s)$ and prove Theorem \ref{thm41} and \ref{thm42}. From \eqref{eqnMrs1}, we have
\begin{equation*}
M(r,s)=\frac{1}{2^{r+s}i^s}\sum_{j=0}^r\sum_{l=0}^s \binom{r}{j}\binom{s}{l}(-1)^{s-l}S(j+l,r+s-(j+l)).
\end{equation*}

When $S(x)$ is complex, we will apply \eqref{eqnlemSa} and \eqref{eqnlemSb} of Lemma \ref{lemSj12}. First suppose that $\psi$ is nontrivial. If $r+s$ is odd, there are no main terms. We have
\begin{align*}
M(r,s)&= O\left(\frac{1}{2^{r+s}}\sum_{j=0}^r\sum_{l=0}^s \binom{r}{j}\binom{s}{l} d^{2(r+s)}H^{r+s}\sqrt{p}\log^{r+s+1}p\right) \\
&= O(d^{2(r+s)}H^{r+s}\sqrt{p}\log^{r+s+1}p).
\end{align*} 
If $r+s=2t$ is even, the main terms are hit when $j+l=(r+s)/2=t$, and the rest of the terms can be estimated as above. In this case we get
\begin{align*}
M(r,s)&= \frac{1}{2^{r+s}i^s}\sum_{\substack{0\leq j\leq r, 0\leq l\leq s \\ j+l=t}} \binom{r}{j}\binom{s}{l}(-1)^{s-l}S(t,t)+O(d^{2(r+s)}H^{r+s}\sqrt{p}\log^{r+s+1}p) \\
&= \frac{t!H^t\abs{\CI}(\beta-\alpha)^t}{2^{r+s}i^s}\sum_{\substack{0\leq j\leq r, 0\leq l\leq s \\ j+l=t}} \binom{r}{j}\binom{s}{l}(-1)^{s-l} \\
&\qquad +O((t!)^2H^{t-1}d^{r+s}\abs{\CI} + 2^t t!d^{2(r+s)}H^{r+s}\sqrt{p}\log^{r+s+1}p).
\end{align*}

Next, if $\psi$ is trivial and $\chi$ has order $a>2$, then we apply \eqref{eqnlemSc}, \eqref{eqnlemSd} and \eqref{eqnlemSe} of Lemma \ref{lemSj12}. If $r+s$ is odd, we have
\begin{equation*}
M(r,s)= O(\left[\frac{r+s-1}{2}\right]!d^{r+s-1}H^{[\frac{r+s-1}{2}]}\abs{\CI} + d^{2(r+s)}H^{r+s}\sqrt{p}\log^{r+s+1}{p}),
\end{equation*}
while if $r+s=2t$ is even, we have
\begin{multline*}
M(r,s)= \frac{t!H^t\abs{\CI}(\beta-\alpha)^t}{2^{r+s}i^s}\sum_{\substack{0\leq j\leq r, 0\leq l\leq s \\ j+l=t}} \binom{r}{j}\binom{s}{l}(-1)^{s-l} \\
+O(t!d^{2t}H^{t-1}\abs{\CI} + d^{2(r+s)}H^{r+s}\sqrt{p}\log^{r+s+1}{p}).
\end{multline*}
Note that the main term here is the same as the case when $\psi$ is nontrivial, and the error term is slightly smaller.

Finally, if $S(x)$ is real, then we apply Lemma \ref{lemSj12real}. We are only interested in $M(k)=M(k,0)$ since we know that $M(r,s)=0$ when $s>0$. We have
\begin{equation*}
M(k)=\frac{1}{2^{k}}\sum_{j=0}^k \binom{k}{j}S(j,k-j).
\end{equation*}
If $k$ is odd, we have
\begin{equation*}
M(k) = O(H^{k}d^{2k}\sqrt{p}\log^{k+1}p),
\end{equation*}
which is the same as the case when $\psi$ is nontrivial. However, when $k$ is even, the main term is different from the other cases. In this case, the main term is contributed from every term since $j-(k-j)=2j-k$ is always even. Therefore we have
\begin{align*}
M(k)&= \frac{1}{2^{k}}\sum_{0\leq j\leq k}\binom{k}{j}\frac{k!}{2^{\frac{k}{2}}\left(\frac{k}{2}\right)!}\abs{\CI}H^{\frac{k}{2}}(\beta-\alpha)^{\frac{k}{2}} \\
&\qquad +O((k/2)!d^{k}H^{k/2-1}\abs{\CI} + d^{2k}H^{k}\sqrt{p}\log^{k+1}{p}) \\
&= \frac{k!}{2^{\frac{k}{2}}\left(\frac{k}{2}\right)!}\abs{\CI}H^{\frac{k}{2}}(\beta-\alpha)^{\frac{k}{2}} +O((k/2)!d^{k}H^{k/2-1}\abs{\CI} + d^{2k}H^{k}\sqrt{p}\log^{k+1}{p}).
\end{align*}

\section{Proof of Theorem \ref{thm1} and Theorem \ref{thm2}}

In this section we will finish off the proof of the main theorems and get the desired limiting Gaussian distribution. The derivation of the Gaussian distribution from the moments are pretty standard (see for example \cite{Fel68}), and we include them here for the sake of completeness. Again we will need to split into two cases depending on whether $S(x)$ is complex or real. We will follow (with modifications to suit our situations) the arguments in \cite{Lam13, MaZa11} below.

\subsection{The case when $S(x)$ is complex}

We first consider the case when $S(x)$ is complex. First we will normalize the sum $S(x)$ by setting
\begin{equation*}
\tilde{S}(x)=\frac{\sqrt{2}}{H^{\frac{1}{2}}(\beta-\alpha)^{\frac{1}{2}}}S(x).
\end{equation*}
Likewise we define the normalized moments
\begin{equation}\label{deftMrs}
\tilde{M}(r,s)=\frac{2^{\frac{r+s}{2}}}{H^{\frac{r+s}{2}}(\beta-\alpha)^{\frac{r+s}{2}}}M(r,s)=\sum_{x\in\CI}(\text{Re}\tilde{S}(x))^r(\text{Im}\tilde{S}(x))^s.
\end{equation}

\begin{remark}
The reason for such a normalization can be seen as follows. When $S(x)$ is complex, by Theorem \ref{thm41}, we have
\begin{equation*}
\frac{1}{\abs{\CI}}\sum_{x\in\CI}\abs{\text{Re}(S(x))}^2 = \frac{1}{\abs{\CI}}\sum_{x\in\CI}\abs{\text{Im}(S(x))}^2\sim\frac{H(\beta-\alpha)}{2}
\end{equation*}
as $p$ tends to infinity. Likewise, when $S(x)$ is real, by Theorem \ref{thm42}, we have
\begin{equation*}
\frac{1}{\abs{\CI}}\sum_{x\in\CI}\abs{S(x)}^2\sim H(\beta-\alpha).
\end{equation*}
This explain why we will need a different normalization for the real case below.
\end{remark}

The main term in Theorem \ref{thm41} (when $r+s=2t$ is even) is
\begin{equation*}
\frac{t!H^t\abs{\CI}(\beta-\alpha)^t}{2^{r+s}i^s}\sum_{\substack{0\leq j\leq r, 0\leq l\leq s \\ j+l=t}} \binom{r}{j}\binom{s}{l}(-1)^{s-l}.
\end{equation*}
Compare with Lemma \ref{lem31}, we obtain
\begin{multline}\label{eqn7a}
M(r,s)=\abs{\CI}(\beta-\alpha)^t E((\text{Re}Z_H)^r(\text{Im}Z_H)^s) \\
+O((t!)^2H^{t-1}d^{r+s}\abs{\CI} + 2^t t!d^{2(r+s)}H^{r+s}\sqrt{p}\log^{r+s+1}p).
\end{multline}

Let 
\begin{equation*}
\phi(u,v)=\frac{1}{\abs{\CI}}\sum_{x\in\CI}e^{iu\text{Re}\tilde{S}(x)+iv\text{Im}\tilde{S}(x)}
\end{equation*}
be the characteristic function of the joint distribution of $\text{Re}\tilde{S}(x)$ and $\text{Im}\tilde{S}(x)$. We will show the following about $\phi(u,v)$.
\begin{prop}\label{prop71}
Let $p$ be a large prime, and let $N$ be a positive integer such that $N\log{N}=O(\log{H})$ and $H^N=O(\abs{\CI}/p^{1/2+\delta})$ for any $\delta>0$ (as long as $\abs{\CI}$ has a higher order than $p^{1/2+\delta}$). Then for any real numbers $u$, $v$ such that $\abs{u},\abs{v}\leq H^{1/4}$, we have
\begin{multline*}
\phi(u,v)=e^{-\frac{u^2+v^2}{2}}\left(1+O\left(\frac{u^4+v^4}{H}\right)\right)+O\left(\frac{2^Nu^{2N}}{N!}+\frac{2^Nv^{2N}}{N!}+\frac{2^{2N}(uv)^{2N}}{(2N)!}\right) \\
+O((2N)!^2d^{4N}H^{-1}+2^NN!d^{4N}H^{4N}\abs{\CI}^{-1}\sqrt{p}\log^{4N}p)(1+u^{2N})(1+v^{2N})).
\end{multline*}
\end{prop}
\begin{proof}
Using the Taylor expansion
\begin{equation*}
e^{ix}=\sum_{j=0}^{2N-1}\frac{(ix)^j}{j!}+O\left( \frac{x^{2N}}{(2N)!} \right), 
\end{equation*}
we have
\begin{equation}\label{eqn7b}
\phi(u,v)=\frac{1}{\abs{\CI}}\sum_{x\in\CI}\sum_{r=0}^{2N-1}\sum_{s=0}^{2N-1}\frac{(iu\text{Re}\tilde{S}(x))^r(iv\text{Im}\tilde{S}(x))^s}{r!s!}+E,
\end{equation}
where the error term $E$ is
\begin{align*}
E &= O\left( \frac{1}{\abs{\CI}}\sum_{x\in\CI}\left( \frac{(u\text{Re}\tilde{S}(x))^{2N}}{(2N)!}+\frac{(v\text{Im}\tilde{S}(x))^{2N}}{(2N)!}+\frac{(uv\text{Re}\tilde{S}(x)\text{Im}\tilde{S}(x))^{2N}}{(2N)!^2} \right) \right) \\
&= O\left( \frac{1}{\abs{\CI}}\left(\frac{u^{2N}}{(2N)!}\tilde{M}(2N,0)+\frac{v^{2N}}{(2N)!}\tilde{M}(0,2N)+\frac{(uv)^{2N}}{(2N)!^2}\tilde{M}(2N,2N)\right) \right).
\end{align*}
By Theorem \ref{thm41} and \eqref{deftMrs}, this gives
\begin{equation}\label{eqn7c}
E = O\left(\frac{2^Nu^{2N}}{N!}+\frac{2^Nv^{2N}}{N!}+\frac{2^{2N}(uv)^{2N}}{(2N)!}\right)
\end{equation}
when $H=o(\log{p})$. For the main term, which we denoted by $\CM$, consider
\begin{align*}
\CM &= \frac{1}{\abs{\CI}}\sum_{x\in\CI}\sum_{r=0}^{2N-1}\sum_{s=0}^{2N-1}\frac{(iu\text{Re}\tilde{S}(x))^r(iv\text{Im}\tilde{S}(x))^s}{r!s!} \\
&= \frac{1}{\abs{\CI}}\sum_{r=0}^{2N-1}\sum_{s=0}^{2N-1}\frac{(iu)^r(iv)^s}{(H/2)^{(r+s)/2}(\beta-\alpha)^{(r+s)/2}r!s!}M(r,s).
\end{align*}
Now invoke \eqref{eqn7a} and \eqref{eqnlem32}, we get
\begin{multline}\label{eqn7d}
\CM = \sum_{r=0}^{2N-1}\sum_{s=0}^{2N-1}\frac{(iu)^r(iv)^s}{r!s!}E((\text{Re}\tilde{Z_H})^r(\text{Im}\tilde{Z_H})^s) \\
+O(((2N)!^2d^{4N}H^{-1}+2^NN!d^{4N}H^{4N}\sqrt{p}\log^{4N}p\abs{\CI}^{-1})(1+u^{2N})(1+v^{2N})).
\end{multline}
Again from the Taylor series, the main term above is
\begin{align*}
&\sum_{r=0}^{2N-1}\sum_{s=0}^{2N-1}\frac{(iu)^r(iv)^s}{r!s!}E((\text{Re}\tilde{Z_H})^r(\text{Im}\tilde{Z_H})^s) \\
=& \left(\sum_{r=0}^{2N-1}\frac{(iu\text{Re}\tilde{Z_H})^r}{r!}\right)\left(\sum_{s=0}^{2N-1}\frac{(iv\text{Im}\tilde{Z_H})^s}{s!}\right) \\
=& E\left( \left(e^{iu\text{Re}\tilde{Z_H}}+O\left(\frac{(u\text{Re}\tilde{Z_H})^{2N}}{(2N)!}\right)\right)\left(e^{iv\text{Im}\tilde{Z_H}}+O\left(\frac{(v\text{Im}\tilde{Z_H})^{2N}}{(2N)!}\right)\right) \right) \\
=& E\left(e^{iu\text{Re}\tilde{Z_H}+iv\text{Im}\tilde{Z_H}}\right)+O\left( \frac{2^Nu^{2N}}{N!}+\frac{2^Nv^{2N}}{N!}+\frac{2^{2N}(uv)^{2N}}{(2N)!} \right),
\end{align*}
where the error term is calculated in the same way as we calculate $E$. Now combining the above equation, \eqref{eqn7b}, \eqref{eqn7c}, \eqref{eqn7d} and Lemma \ref{lem32} yields Proposition \ref{prop71}.
\end{proof}

Theorem \ref{thm1} now follows from Proposition \ref{prop71} by exactly the same arguments in \cite[Section 4]{Lam13}, which is also the method that Selberg used in his proof that $\log\zeta(1/2+it)$ has a limiting two-dimensional Gaussian distributon (see \cite{Tsa84,Sel92}).

\subsection{The case when $S(x)$ is real}

When $S(x)$ is real, we will need a different normalization in the form of
\begin{equation*}
S_R(x)=\frac{1}{H^{\frac{1}{2}}(\beta-\alpha)^{\frac{1}{2}}}S(x),
\end{equation*}
and
\begin{equation*}
M_R(k)=\sum_{x\in\CI}S_R(x)^k.
\end{equation*}
The main term in Theorem \ref{thm42} (for $k$ even) is
\begin{equation*}
\frac{k!}{2^{\frac{k}{2}}\left(\frac{k}{2}\right)!}\abs{\CI}H^{\frac{k}{2}}(\beta-\alpha)^{\frac{k}{2}} = 1\cdot 3\cdot\ldots\cdot (k-1)\cdot \abs{\CI}\cdot H^{\frac{k}{2}}(\beta-\alpha)^{\frac{k}{2}}.
\end{equation*}
when $k$ is even. Write
\begin{equation*}
\mu_k = 
\begin{cases}
1\cdot 3\cdot\ldots\cdot (k-1) &, k \text{~even}, \\
0 &, k \text{~odd}.
\end{cases}
\end{equation*}
Then we have
\begin{equation*}
\lim_{p\rightarrow\infty}\frac{M_R(k)}{\abs{\CI}}=\mu_k.
\end{equation*}
That is
\begin{equation}\label{eqn71}
\lim_{p\rightarrow\infty}\frac{1}{\abs{\CI}}\sum_{x\in\CI}S_R(x)^k=\mu_k.
\end{equation}

Recall that $G_p(\lambda)$ is the number of $x\in\CI$ such that $S(x)\leq \lambda(H(\beta-\alpha))^{1/2}$, i.e. the number of $x\in\CI$ with $S_R(x)\leq\lambda$. Then $G_p(\lambda)$ is a monotonic increasing step-function of $\lambda$, with discontinuities at $\lambda=\lambda_1,\ldots,\lambda_h$, say. Note that $G_p(\lambda)=0$ if $\lambda<-H$, and $G_p(\lambda)=\abs{\CI}$ if $\lambda\geq H$. Collect together the values of $x\in\CI$ for which $S_R(x)=\lambda_i$, then \eqref{eqn71} gives (with the convention $G_p(\lambda_0)=0$)
\begin{equation*}
\lim_{p\rightarrow\infty}\frac{1}{\abs{\CI}}\sum_{j=1}^h (\lambda_j)^k (G_p(\lambda_j)-G_p(\lambda_{j-1})) = \mu_k.
\end{equation*}
The LHS in the above equation can be written as a Riemann-Stieltjes integral
\begin{equation*}
\frac{1}{\abs{\CI}}\sum_{j=1}^h (s_j)^k (G_p(\lambda_j)-G_p(\lambda_{j-1})) = \int_{-\infty}^{\infty}t^k \,d\phi_p(t),
\end{equation*}
where
\begin{equation*}
\phi_p(t) = \frac{1}{\abs{\CI}}G_p(\lambda).
\end{equation*}

On the other hand, set
\begin{equation*}
\phi(t)=\frac{1}{\sqrt{2\pi}}\int_{-\infty}^t e^{-u^2/2}\,du,
\end{equation*}
that is, the distribution function for the standard Gaussian distribution. Then we have
\begin{equation*}
\int_{-\infty}^{\infty}t^k\,d\phi(t)=\frac{1}{\sqrt{2\pi}}\int_{-\infty}^{\infty}t^k e^{-t^2/2}\,dt=\mu_k
\end{equation*}
for all $k$. From this, one can deduce from probability theory (see \cite{Fel68}) that
\begin{equation*}
\lim_{p\rightarrow\infty}\phi_p(t) = \phi(t).
\end{equation*}
This proves Theorem \ref{thm2}.

\subsection*{Acknowledgement}

The author would like to thank Douglas Ulmer for many valuable discussions on the algebraic geometry involved in this paper. In particular a strategy suggested by him led to the proof of Proposition \ref{prop51}.


\end{document}